\documentclass[11pt]{amsart}
\usepackage{amssymb}
\usepackage{amsfonts}
\usepackage{mdwlist}
\usepackage{latexsym}
\usepackage{amsmath}
\usepackage{amsfonts}
\usepackage{enumitem}
\usepackage{fancyhdr}
\usepackage{cleveref}
\usepackage{mathrsfs}
\usepackage{url}
\usepackage{amscd}
\usepackage{color}

\theoremstyle{plain}
\newtheorem{theorem}{Theorem}[section]

\newtheorem{lemma}[theorem]{Lemma}

\newtheoremstyle{remark}
    {} 
    {} 
    {}          
    {}          
    {\bfseries} 
    {.}         
    {.5em}      
    {}          
\theoremstyle{remark}

\newtheoremstyle{example}
    {\dimexpr\topsep/2\relax} 
    {\dimexpr\topsep/2\relax} 
    {}          
    {}          
    {\bfseries} 
    {.}         
    {.5em}      
    {}          
\theoremstyle{example}

\newtheoremstyle{definition}
    {\dimexpr\topsep/2\relax} 
    {\dimexpr\topsep/2\relax} 
    {}          
    {}          
    {\bfseries} 
    {.}         
    {.5em}      
    {}          
\theoremstyle{plain}
\newtheorem{thm}{Theorem}[section]

\newtheorem{defn}[thm]{Definition}

\numberwithin{equation}{section}

\allowdisplaybreaks

\begin{document}
\title{A non-abelian, non-Sidon, completely bounded $\Lambda(p)$ set}
\author{Kathryn E. Hare}
\address{Dept. of Pure Mathematics\\
University of Waterloo\\
Waterloo, Ont., \\
Canada}
\email{kehare@uwaterloo.ca}
\thanks{This research was supported in part by NSERC\ grant RGPIN 2016-03719}
\author{Parasar Mohanty}
\address{Dept. of Mathematics and Statistics\\
Indian Inst. of Tech.\\
Kanput, India, 208016}
\email{parasar@iitk.ac.in}

\begin{abstract}
The purpose of this note is to construct an example of a discrete
non-abelian group $G$ and a subset $E$ of $G$, not contained in any abelian
subgroup, that is a completely bounded $\Lambda (p)$ set for all $p<\infty ,$
but is neither a Leinert set nor a weak Sidon set.
\end{abstract}

\maketitle



\section{Introduction}

The study of lacunary sets, such as Sidon sets and $\Lambda (p)$ sets,
constitutes an interesting theme in the theory of Fourier series on the
circle group ${\mathbb{T}}$. It has many applications in harmonic analysis
and in the theory of Banach spaces, and various combinatorial and arithmetic
properties of these sets have been studied extensively. These concepts have
also been investigated in the context of more general compact abelian groups
(with their discrete dual groups) and compact non-abelian groups; see \cite%
{GH}, \cite{LR}, \cite{R} and the references cited therein. The study of
these sets in the setting of discrete non-abelian groups was pioneered by
Bozjeko \cite{B}, Fig\'{a}-Talamanca \cite{FP} and Picardello \cite{Pi}.

In abelian groups, there are various equivalent ways to define Sidon sets
and these sets are plentiful. Indeed, every infinite subset of a discrete
abelian group contains an infinite Sidon set. The natural analogues of these
definitions in discrete non-abelian groups are known as strong Sidon, Sidon
and weak Sidon sets. It was shown in \cite{Pi} that every weak Sidon set is $%
\Lambda (p)$ for all $p<\infty $. In \cite{Le} Leinert introduced the
concept of a $\Lambda (\infty )$ set, a notion only of interest in the
non-abelian setting because in abelian groups such sets are necessarily
finite. In striking contrast to the abelian situation, Leinert showed that
the free group with two generators contains an infinite subset which is both
weak Sidon and $\Lambda (\infty ),$ but does not contain any infinite Sidon
subsets.

In \cite{Ha}, Harcharras studied the concept of completely bounded $\Lambda
(p)$ sets, a property more restrictive than $\Lambda (p),$ but still
possessed by Sidon sets. The converse is not true as every infinite discrete
abelian group admits a completely bounded $\Lambda (p)$ set which is not
Sidon; see \cite{HM}.

In this paper, we construct a non-amenable group $G$ and a set $E$ not
contained in any abelian subgroup of $G,$ which is completely bounded $%
\Lambda (p)$ for every $p<\infty ,$ but is neither $\Lambda (\infty )$ nor
weak Sidon. It remains open if every infinite discrete group contains such a
set $E$.

\section{Definitions}

Throughout this paper, $G$ will be an infinite discrete group. To define
Sidon and $\Lambda (p)$ sets in this setting one requires the concepts of
the Fourier algebra, $A(G)$, the von Neumann algebra, $VN(G),$ and the
Fourier-Stieljies algebra, $B(G)$, as developed by P. Eymard in \cite{E} for
locally compact groups. We also need the concept of a non-commutative $L^{p}$%
-spaces introduced by I.E. Segal. We refer the reader to \cite{PX} for
details on these latter spaces.


\begin{defn}
(i) The set $E\subseteq $ $G$ is said to be a \textbf{strong (weak) Sidon set%
} if for all $f\in c_{0}(E)$ (resp., $l_{\infty }(E))$ there exists $g\in
A(G)$ (resp., $B(G))$ such that $f(x)=g(x)\;$for all $x\in E$.

(ii) The set $E\subseteq G$ is said to be a \textbf{Sidon set }if there is a
constant $C$ such that for all functions $f,$ compactly supported in $E,$ we
have $\Vert f\Vert _{1}\leq C\Vert f\Vert _{VN(G)}$. The least such constant 
$C$ is known as the \textbf{Sidon constant} of $E$.
\end{defn}


These definitions are well known to be equivalent in the commutative
setting. For any discrete group it is the case that strong Sidon sets are
Sidon and Sidon sets are weak Sidon. Finite groups are always strong Sidon
sets. In \cite{Pi} it is shown that $E\subseteq G$ is Sidon if and only if
for every $f\in $ $l_{\infty }(E)$ there is some $g\in B_{\rho }(G)$ that
extends $f$, where $B_{\rho }(G)$ is the dual of the reduced $C^{\ast }$
algebra $C_{\rho }^{\ast }(G)$. Since in an amenable group $B_{\rho
}(G)=B(G),$ weak Sidon sets are Sidon in this setting. Very recently, Wang 
\cite{Wa} showed that every Sidon set in any discrete group is a strong
Sidon set. It remains open if every infinite amenable group contains an
infinite Sidon subset.\textit{\ }

Picardello \cite{Pi} defined the notion of $\Lambda (p)$ sets in this
setting and Harcharras \cite{Ha} introduced completely bounded $\Lambda (p)$
sets. For these, we require further notation. Let $\leftthreetimes $ denote
the left regular representation of $G$ into $\mathcal{B}(l_{2}(G))$ and
denote by $L^{p}(\tau _{0})$ the non-commutative $L^{p}$-space associated
with the von Neumann algebra generated by $\leftthreetimes (G)$ with respect
to the usual trace $\tau _{0}$. Let $L^{p}(\tau )$ denote the
non-commutative $L^{p}$-space associated with the von Neumann algebra
generated by $\leftthreetimes (G)\otimes \mathcal{B}(l_{2})$ with respect to
the trace $\tau =\tau _{0}\otimes tr$, where $tr$ denotes the usual trace in 
$\mathcal{B}(l_{2})$. Observe that $L^{p}(\tau )$ has a cannonical operator
space structure obtained from complex interpolation in the operator space
category. We refer the reader to \cite{Pis} for more details.

\begin{defn}
(i) Let $2<p<\infty $. The set $E\subseteq G$ is said to be a $\Lambda (p)$%
\textbf{\ set} if there exists a constant $C_{1}>0$ such that for all
finitely supported functions $f$ we have 
\begin{equation}
\left\Vert \sum\limits_{t\in E}f(t)\leftthreetimes (t)\right\Vert
_{L^{p}(\tau _{0})}\leq C_{1}\left( |\sum\limits_{t\in E}|f(t)|^{2}\right) ^{%
\frac{1}{2}}.  \label{Lambdap}
\end{equation}

(ii) The set $E\subseteq G$ is said to be a \textbf{completely bounded }$%
\Lambda (p)$\textbf{\ set, }denoted $\Lambda ^{cb}(p),$ if there exists a
constant $C_{2}>0$ such that 
\begin{equation}
\Vert \sum\limits_{t\in E}\leftthreetimes (t)\otimes x_{t}\Vert _{L^{p}(\tau
)}\leq C_{2}\max \left( \Vert (\sum_{t\in E}x_{t}^{\ast }x_{t})^{1/2}\Vert
_{S_{p}},\Vert (\sum_{t\in E}x_{t}x_{t}^{\ast })^{1/2}\Vert _{S_{p}}\right)
\label{CBLp}
\end{equation}%
where $x_{t}$ are finitely supported families of operators in $S_{p}$, the $%
p $-Schatten class on $l_{2}$.

The least such constants $C_{1}$ (or $C_{2})$ are known as the $\Lambda (p)$%
\textbf{\ }(resp.,\textbf{\ }$\Lambda ^{cb}(p)$\textbf{) constants} of $E$.
\end{defn}

It is known that every infinite set contains an infinite $\Lambda (p)$ set 
\cite{B} and that every weak Sidon set is a $\Lambda (p)$ set for each $%
p<\infty $ \cite{Pi}. \ Completely bounded $\Lambda (p)$ sets are clearly $%
\Lambda (p),$ but the converse is not true, as seen in \cite{Ha}.

Extending these notions to $p=\infty $ gives the Leinert and $L$-sets.

\begin{defn}
(i) The set $E\subseteq $ $G$ is called a \textbf{Leinert or }$\Lambda
(\infty )$\textbf{\ set} if there exists a constant $C>0$ such that for
every function $f\in l_{2}(E)$ we have $\Vert f\Vert _{VN(G)}\leq C\Vert
f\Vert _{2}$.

(ii) The sets of interpolation for the completely bounded multipliers of $%
A(G)$ are called\textbf{\ }$L$\textbf{-sets}.
\end{defn}

It is well known that the Leinert sets are the sets of interpolation for
multipliers of $A(G),$ so any $L$-set is Leinert; see \cite{Po}. The set $E$
is said to satisfy the \textbf{Leinert condition} if every tuple $%
(a_{1},...,a_{2s})\in E^{2s},$ with $a_{i}\neq a_{i+1},$ satisfies the
independence-like relation 
\begin{equation}
a_{1}a_{2}^{-1}a_{3}\dots a_{2s-1}a_{2s}^{-1}\neq e.  \label{leinert}
\end{equation}%
Here $e$ is the identity of $G$. It can be shown (\cite{Po}) that any set
that satisfies the Leinert condition is an $L$-set.

It was seen in \cite{HM} that in abelian groups there are sets that are
completely bounded $\Lambda (p)$ for all $p<\infty ,$ but not Sidon. Thus
the inclusion, weak Sidon is $\Lambda ^{cb}(p),$ is strict for groups with
infinite abelian subgroups. The purpose of this paper is to show the
existence of sets not contained in \textit{any} abelian subgroup which also
have this strict inclusion. In fact, we prove, more generally, the following
result.

\begin{theorem}
There is a discrete group $G$ that admits both infinite $L$-sets and weak
Sidon sets, and an infinite subset $E$ of $G$ that is $\Lambda ^{cb}(p)$ for
all $p<\infty ,$ but not a Leinert set, an $L$-set or a weak Sidon set.
Moreover, any subset of $E$ consisting of commuting elements is finite. %
\label{mainthm}
\end{theorem}

\section{Results and Proofs}



\subsection{Preliminary results}

To show that the set we will construct is not a Leinert or weak Sidon set,
it is helpful to first establish some arithmetic properties of $\Lambda (p)$
and Leinert sets. We recall that a set $E\subseteq G$ is said to be \textbf{%
quasi-independent} if all the sums 
\begin{equation*}
\left\{ \sum\limits_{x\in A}x:A\subset E,|A|<\infty \right\}
\end{equation*}%
are distinct. Quasi-independent sets in abelian groups are the prototypical
Sidon sets.

The first part of the following Lemma is well known for abelian groups.

\begin{lemma}
Let $G$ be a discrete group. (i) Suppose $q>2$ and $E\subseteq G$ is a $%
\Lambda (q)$ set with $\Lambda (q)$ constant $A$. If $a\in G$ has order $%
p_{n}\geq $ $2n$, then 
\begin{equation*}
\left\vert E\bigcap \{a,a^{2},...,a^{n}\}\right\vert \leq 10A^{2}n^{2/q}%
\text{.}
\end{equation*}

(ii) Suppose $E\subseteq G$ is a Leinert set with Leinert constant $B$ and
let $F\subseteq E$ be a finite commuting, quasi-independent subset. Then $%
\left\vert F\right\vert \leq 6^{3}B^{2}$. \label{mainlem}
\end{lemma}

\begin{proof}
We will write $1_{X}$ for the characteristic function of a set $X$.

(i) Define the function $K_{n}$ on $G$ by 
\begin{equation*}
K_{n}(x)=\sum_{j=-2n}^{2n}\left( 1-\frac{\left\vert j\right\vert }{n}\right)
1_{\{a^{j}\}}(x).
\end{equation*}%
Let $J_{n}$ denote the function on $\mathbb{Z}_{p_{n}}$ (or $\mathbb{Z}$ if $%
p_{n}=\infty $) defined in the analogous fashion. It is well known that the $%
A(G)$ and $VN(G)$ norms for the function $K_{n}$ are dominated by the
corresponding norms of the function $J_{n}$ on $\mathbb{Z}_{p_{n}}$.



As $L^{q^{\prime }}(\tau _{0})$ (for $q^{\prime }$ the dual index to $q$) is
an interpolation space between $A(G)$ and $VN(G)$, it follows that%
\begin{eqnarray*}
\left\Vert K_{n}\right\Vert _{L^{q^{\prime }}(\tau )} &\leq &\left\Vert
K_{n}\right\Vert _{A(G)}^{1/q^{\prime }}\left\Vert K_{n}\right\Vert
_{VN(G)}^{1/q} \\
&=&\left\Vert J_{n}\right\Vert _{A(\mathbb{Z}_{p_{n}})}^{1/q^{\prime
}}\left\Vert J_{n}\right\Vert _{VN(\mathbb{Z}_{p_{n}})}^{1/q}\leq
(4n+1)^{1/q}\text{.}
\end{eqnarray*}

Suppose $E\bigcap \{a,a^{2},...,a^{n}\}$ consists of the $M$ elements $%
\{a^{s_{j}}\}_{j=1}^{M}$ and put 
\begin{equation*}
k_{n}(x)=\sum_{j=1}^{M}1_{\{a^{s_{j}}\}}(x).
\end{equation*}%
Since $E$ has $\Lambda (q)$ constant $A,$ the generalized Holder's
inequality implies%
\begin{eqnarray*}
\frac{M}{2} &\leq &\sum_{j=1}^{M}K_{n}(a^{s_{j}})=\sum_{x\in
G}K_{n}(x)k_{n}(x) \\
&\leq &\left\Vert K_{n}\right\Vert _{L^{q^{\prime }}(\tau _{0})}\left\Vert
k_{n}\right\Vert _{L^{q}(\tau _{0})}\leq (4n+1)^{1/q}A\left\Vert
k_{n}\right\Vert _{2} \\
&=&(4n+1)^{1/q}A\sqrt{M}.
\end{eqnarray*}%
Consequently, $M\leq 2(4n+1)^{2/q}A^{2}\leq 10A^{2}n^{2/q}$, as claimed.

(ii) Let $H$ be the abelian group generated by $F$. Being quasi-independent, 
$F$ is a Sidon subset of $H$ with Sidon constant at most $6\sqrt{6}$ (\cite[%
p.115]{GH}). Consider the function $h=1_{F}$ defined on $H$ and $g=1_{F}$
defined on $G$. The Sidon property, together with the fact that $\left\Vert
h\right\Vert _{VN(H)}=\left\Vert g\right\Vert _{VN(G)},$ ensures that 
\begin{equation*}
\left\vert F\right\vert =\left\Vert h\right\Vert _{\ell ^{1}}\leq 6\sqrt{6}%
\left\Vert h\right\Vert _{VN(H)}=6\sqrt{6}\left\Vert g\right\Vert _{VN(H)}.
\end{equation*}%
Since $E$ has Leinert constant $B$, we have $\left\Vert f\right\Vert
_{VN(G)}\leq B\left\Vert f\right\Vert _{2}$ for any function $f$ defined on $%
G$ and supported on $E$. In particular, this is true for the function $g\,$,
hence 
\begin{equation*}
\left\vert F\right\vert \leq 6\sqrt{6}\left\Vert g\right\Vert _{VN(H)}\leq 6%
\sqrt{6}B\sqrt{\left\vert F\right\vert }.
\end{equation*}
\end{proof}

\subsection{\noindent Proof of Theorem \protect\ref{mainthm}}

\begin{proof}
We will let $G$ be the free product of the cyclic groups $Z_{p_{n}}$, $n\in
N $, where $p_{n}>2^{n+1}$ are distinct odd primes. If $a_{n}$ is a
generator of $Z_{p_{n}},$ then $\{a_{n}\}_{n=1}^{\infty }$ is both a weak
Sidon and Leinert set, as shown in \cite{Pi}. The set $E$ will be the union
of finite sets $E_{n}\subseteq Z_{p_{n}}$, where $\left\vert
E_{n}\right\vert =n^{2}$ and $E_{n}\subset \{a_{n},...,a_{n}^{2^{n}}\}$. The
fact that any commuting subset of $E$ is finite is obvious from the
definition of $E$.\newline

We recall the following notation from \cite{Ha}: We say that a subset $%
\Lambda \subseteq G$ has the $Z(p)$ property if $Z_{p}(\Lambda )<\infty $
where%
\begin{equation*}
Z_{p}(\Lambda )=\sup_{x\in G}\left\vert \left\{ (x_{1},...,x_{p})\in \Lambda
^{p}:x_{i}\neq x_{j},x_{1}^{-1}x_{2}x_{3}^{-1}\cdot \cdot \cdot
x_{p}^{(-1)^{p}}=x\right\} \right\vert .
\end{equation*}%
In \cite{Ha}, Harcharras proved that if $2<p<\infty ,$ then every subset $%
\Lambda $ of $G$ with the $Z(p)$ property is a $\Lambda ^{cb}(2p).$

We will construct the sets $E_{n}$ so that they have the property that for
every even $s\geq 2$ there is an integer $n_{s}$ such that $Z_{s}\left(
\bigcup_{n\geq n_{s}}E_{n}\right) \leq s!$. Consequently, $\bigcup_{n\geq
n_{s}}E_{n}$ will be $\Lambda ^{cb}(2s)$ for all $s<\infty $. As finite sets
are $\Lambda ^{cb}(p)$ for all $p<\infty ,$ and a finite union of $\Lambda
^{cb}(p)$ sets is again $\Lambda ^{cb}(p)$, it will follow that $E$ is $%
\Lambda ^{cb}(p)$ for all $p<\infty $.

We now proceed to construct the sets $E_{n}$ by an iterative argument.
Temporarily fix $n$ and take $g_{1}=a_{n}$. Inductively assume that for $%
N<n^{2}$, $\{g_{i}\}_{i=1}^{N}$ $\subseteq \{a_{n},...,a_{n}^{2^{n}}\}$ have
been chosen with the property that if 
\begin{equation}
\prod\limits_{j=1}^{N}g_{j}^{\varepsilon _{j}}=1\text{ for }\varepsilon
_{j}=0,\pm 1,\pm 2, \sum_{j}|\varepsilon _{j}|\leq 2s,\text{ then all }%
\varepsilon _{j}=0.  \tag{$\mathcal{P}_{N}$}
\end{equation}%
Now choose 
\begin{equation*}
g_{N+1}\neq \prod\limits_{j=1}^{N}g_{j}^{\varepsilon _{j}}\text{ for any }%
\varepsilon _{j}=0,\pm 1,\pm 2\text{ and }\sum_{j}|\varepsilon _{j}|\leq 2s
\end{equation*}%
and 
\begin{equation*}
g_{N+1}^{2}\neq \prod\limits_{j=1}^{N}g_{j}^{\varepsilon _{j}}\text{ for any 
}\varepsilon _{j}=0,\pm 1,\pm 2\text{ and }\sum_{j}|\varepsilon _{j}|\leq 2s.
\end{equation*}%
There are at most $\binom{N}{2s}5^{2s}\leq C_{s}N^{2s}$ terms that $g_{N+1}$
must avoid and similarly for $g_{N+1}^{2}$ as the squares of elements of $%
Z_{p_{n}}$ are all distinct. Provided $2C_{s}N^{2s}\leq 2^{n}$ then we can
make such a choice of $g_{N+1}\in $ $\{a_{n},...,a_{n}^{2^{n}}\}$. Of
course, it is immediate that property ($\mathcal{P}_{N+1}$) then holds. This
can be done for every $N<n^{2}$ as long as $n$ is suitably large, say for $%
n\geq n_{s}$. The set $E_{n}$ will be taken to be $\{g_{j}\}_{j=1}^{n^{2}%
\text{.}}.$

Now we need to check the claim that $Z_{s}(\bigcup\limits_{n\geq
n_{s}}E_{n})\leq s!$. Towards this, suppose 
\begin{equation}
x_{1}x_{2}^{-1}\cdot \cdot \cdot x_{s}^{-1}=y_{1}y_{2}^{-1}\cdot \cdot \cdot
y_{s}^{-1}  \label{P1}
\end{equation}%
where $x_{i}$ are all distinct, $y_{j}$ are all distinct and all $%
x_{i},y_{j}\in \bigcup\limits_{n\geq n_{s}}E_{n}$. The free product property
guarantees that if this is true, then it must necessarily be the case that
if we consider only the elements $x_{i_{k}}$ and $y_{j_{l}}$ which belong to
a given $E_{n}$, we must have $\prod\limits_{k}x_{i_{k}}^{\delta _{k}}=$ $%
\prod\limits_{l}y_{j_{l}}^{\varepsilon _{l}}$ for the appropriate choices of 
$\delta _{k},\varepsilon _{l}\in \{\pm 1\}$. As there at most $s$ choices
for each of $x_{i_{k}}$ and $y_{i_{l}}$, our property ($\mathcal{P}_{N}$)
ensures that this can happen only if $\{x_{i_{k}}:\delta _{k}=1\}$ $%
=\{y_{j_{l}}:\varepsilon _{l}=1\}$ and similarly for the terms with $-1$
exponents. Hence we can only satisfy (\ref{P1}) if upon reordering, $%
\{x_{1},x_{3},...,x_{s-1}\}=\{y_{1},y_{3},...,y_{s-1}\},$ and similarly for
the terms with even labels. (We remark that for non-abelian groups, this is
only a necessary but not, in general, sufficient condition for (\ref{P1}).)
This suffices to establish that 
\begin{equation*}
Z_{s}(\bigcup\limits_{n\geq n_{s}}E_{n})\leq ((s/2)!)^{2}\leq s!
\end{equation*}%
and hence, as explained above, $E$ is a $\Lambda ^{cb}(p)$ set for all $%
p\,<\infty $.

Next, we will verify that $E$ is not a weak Sidon set. We proceed by
contradiction. According to \cite{Pi}, if it was, then $E$ would be a $%
\Lambda (p)$ set for each $p>2,$ with $\Lambda (p)$ constant bounded by $C%
\sqrt{p}$ for a constant $C$ independent of $p$. Appealing to Lemma \ref%
{mainlem}(i), we have%
\begin{equation*}
n^{2}=\left\vert E_{n}\right\vert =\left\vert E\bigcap
\{a_{n},...,a_{n}^{2^{n}}\}\right\vert \leq 10C^{2}p2^{2n/p}.
\end{equation*}%
Taking $p=2n$ for sufficiently large $n$ gives a contradiction.

Finally, to see that $E$ is not a Leinert set, we first observe that an easy
combinatorial argument shows that any set of $N$ distinct elements contains
a quasi-independent subset of cardinality at least $\log N/\log 3$. Thus we
can obtain quasi-independent subsets $F_{n}\subseteq E_{n}$ with $\left\vert
F_{n}\right\vert \rightarrow \infty $. But according to Lemma \ref{mainlem}%
(ii), this would be impossible if $E$ was a Leinert set. As $E$ is not
Leinert, it is also not an $L$ set.

This concludes the proof.
\end{proof}

\

\end{document}